\documentclass[11pt]{amsart}
\usepackage{amsmath}
\usepackage{amssymb}
\usepackage{graphicx} 

\usepackage{hyperref}
\usepackage{caption, subcaption}
\usepackage{pgfplots}

\usepackage{setspace}
\DeclareFontFamily{OT1}{pzc}{}
\DeclareFontShape{OT1}{pzc}{m}{it}%
              {<-> s * [1.2] pzcmi7t}{}
\DeclareMathAlphabet{\mathpzc}{OT1}{pzc}%
                                 {m}{it}

\newtheorem{thm}{Theorem}[section]
\newtheorem*{unnamedthm}{Theorem}

\newtheorem{cor}[thm]{Corollary}

\newcommand{\Z}{\mathbb{Z}}      
\newcommand{\Q}{\mathbb{Q}}      
\newcommand{\R}{\mathbb{R}}			 

\newcommand{\Qfield}[1]{\Q(\sqrt{#1})}

\newcommand{\OQfield}[1]{\mathcal{O}_{\Qfield{#1}}}

\title{Walks on Primes in Imaginary Quadratic Fields}

\author{Siddharth Prasad}
\email{spmdccxxix (at) gmail (dot) com}

\begin{document}

\begin{abstract}
We generalize the Gaussian moat problem to the finitely many imaginary quadratic fields exhibiting unique factorization in their integers. We map this problem to the Euclidean minimum spanning tree problem and use Delaunay triangulation and Kruskal's algorithm to find moats  in these fields in $O(|\mathbf{P}| \log |\mathbf{P}|)$ time, where $\mathbf{P}$ denotes the set of primes we are searching through. Our algorithm uses  symmetry and boundaries to correctly and efficiently identify all moats, in order of increasing length, along with the farthest prime reached for each moat, in an underlying unbounded graph of infinite primes. We also derive asymptotic estimates on the density of primes in certain imaginary quadratic fields. Based on this and symmetry arguments, we conjecture that for a given step bound $k$, we can travel farther away from the origin in $\Z[i]$ and $\Z\left[\frac{-1+\sqrt{-3}}{2}\right]$ than in $\Z[\sqrt{-2}]$. This conjecture is corroborated by our computational results.
\end{abstract}

\maketitle

\section{Introduction}

\subsection{Moat Problem}
The integer moat problem asks whether it is possible to walk to infinity by taking steps of bounded length on prime numbers in $\Z$ (called rational primes). This is not possible since $\{(n+1)!+i\}_{i=2}^{i=n+1}$ gives $n$ consecutive composite integers. Equivalently, $\lim\sup\{p_{n+1}-p_{n}\}=\infty$.

The Gaussian moat problem deals with a similar walk to infinity by taking steps on Gaussian primes, which are defined below. In general, the existence of a $k$-moat refers to the fact that it is not possible to walk to infinity with step size at most $k$ (measured by distance on the complex plane). The Gaussian moat problem, which asks whether or not there exist $k$-moats in the Gaussian integers for arbitrarily large $k$ was posed by Basil Gordon in 1962 and remains unsolved.

The first computational exploration of this problem was by Jordan and Rabung in 1969, who constructed $1$, $\sqrt{2}$, $2$, $\sqrt{8}$ and $\sqrt{10}$-moats~\cite{jordan}. In 1998, Gethner et. al. made further computational progress by constructing a $4$ and $\sqrt{18}$-moat and showing the existence of a $\sqrt{26}$-moat~\cite{gethner}. Most recently, in 2004, Tsuchimura further established the existence of $\sqrt{32}$, $\sqrt{34}$, and $6$-moats~\cite{japan}. 

Although the original question remains unsolved, the Gaussian moat problem has spurred interesting discussion on other relevant questions. For example, \cite{gethner} shows that there does not exist a bounded walk to infinity on Gaussian primes constrained to any single line in the complex plane.

\subsection{Quadratic Fields}
A quadratic field is a degree $2$ extension of $\Q$. Let $d\neq 1$ be a square-free integer. Then, the sets of the form $$\Qfield{d}=\{a+b\sqrt{d}:a,b\in\Q\}$$ are the quadratic fields with $\Q$-basis $\{1,\sqrt{d}\}$. Let $\OQfield{d}\subset \Qfield{d}$ denote the ring of algebraic integers in $\Q(\sqrt{d})$. The following well known theorem characterizes this ring of integers.

\begin{unnamedthm}If $d\equiv 1\pmod{4}$, then $\OQfield{d}=\Z\Bigl[\frac{-1+\sqrt{d}}{2}\Bigr]$. Otherwise, $\OQfield{d}=\Z[\sqrt{d}]$. \end{unnamedthm}

The elements that make up the set $\OQfield{-1}=\Z[i]$ are known as the Gaussian integers, and those that make up the set $\OQfield{-3}=\Z\Bigl[\frac{-1+\sqrt{-3}}{2}\Bigr]$, commonly denoted as $\Z[\omega]$ are known as the Eisenstein integers.

$\Q[\sqrt{d}]$ does not admit a well defined notion of primality, since based on the usual definitions all non-zero elements of $\Qfield{d}$ are trivially prime and reducible.\footnote{In a set $S$, $p\in S$ is a prime element if $p|ab\Rightarrow p|a$ or $p|b$ for $a,b\in S$. An irreducible element is one that cannot be expressed as a product of non-units in $S$. In unique factorization domains, such as $\Z$ and those discussed later, prime elements and irreducible elements are identical.} Hence, we will ignore $\Qfield{d}$, and only study the associated integer ring $\OQfield{d}$. Henceforth, all discussion of any quadratic field refers to the associated integer ring.

We are specifically interested in imaginary quadratic fields ($d<0$), since these exhibit intuitive geometric properties based on the norm which is defined as $N(\alpha)=\alpha\bar\alpha$, where $\bar\alpha$ is the conjugate of $\alpha$ in $\Qfield{d}$. The norm is a multiplicative function. If $d\equiv 1\pmod{4}$, $N\Bigl(x+y\Bigl(\frac{-1+\sqrt{d}}{2}\Bigr)\Bigr)=x^2-xy-\frac{d-1}{4}y^2$, and if $d\equiv 2\pmod{4}$ or $d\equiv 3\pmod{4}$, then $N(x+y\sqrt{d})=x^2-dy^2$. For $d<0$, the norm is always positive. In particular, we study a subset of the imaginary quadratic fields that are unique factorization domains, proved by H. Stark to be those with $d=-1$, $-2$, $-3$, $-7$, $-11$, $-19$, $-43$, $-67$, and $-163$~\cite{ireland}.

\subsection{Our Contribution}
In this paper, we generalize the Gaussian moat problem to the finitely many imaginary quadratic fields exhibiting unique factorization in their integers. We also propose a new and efficient algorithm by mapping this problem to the Euclidean minimum spanning tree problem and using Delaunay triangulation and Kruskal's algorithm to find moats in these fields in $O(|\mathbf{P}|\log |\mathbf{P}|)$ time, where $\mathbf{P}$ denotes the set of primes we are searching through. While the underlying algorithms -- including those to identify the triangulation and the minimum spanning tree to find a path with the minimum maximal edge length -- are known, our algorithm appears to be the first to apply it explicitly to find all moats, in increasing order of length, along with the associated farthest prime. Further to deal with an underlying infinite unbounded graph with infinite nodes (primes), the algorithm uses symmetry and boundaries to correctly and efficiently identify moats.


We also derive asymptotic estimates on the density of primes in certain imaginary quadratic fields. Based on this and symmetry arguments we make, we conjecture that for a given step bound $k$, we can travel farther away from the origin in $\Z[i]$ and $\Z[\omega]$ than in $\Z[\sqrt{-2}]$. This conjecture is corroborated by our computational results and the fact that $\Z[i]$ and $\Z[\omega]$ have smaller discriminants than that of $\Z[\sqrt{-2}]$, since the former two rings contain a larger number of algebraic integers within a certain distance (and thus likely contain more primes).

\section{Methods}

\subsection{Prime Generation}

In order to generate primes in $\OQfield{d}$, we shall use the following important classifications of primes, the proofs of which can be found in \cite{ireland}. Let $\pi\in\OQfield{d}$.

\begin{thm}If $N(\pi)$ is a rational prime, $\pi$ is prime. \label{normprime} \end{thm}
Suppose $p$ is a rational prime. We say $p$ is inert if $p$ is also prime in $\OQfield{d}$. In order to classify which rational primes remain inert, we define the discriminant of a quadratic field, $\delta=\begin{cases}d & \text{ if } d\equiv 1\pmod{4} \\ 4d & \text{ else}\end{cases}$
\begin{thm}An odd rational prime $p$ is inert if $p\nmid\delta$ and $\left(\frac{\delta}{p}\right)=-1$. \label{inertprimes} \end{thm}

Finally, we deal with the rational prime $2$ below.

\begin{thm}$2$ is prime in $\OQfield{d}$ if $2\nmid \delta$ and $d\equiv 5\pmod{8}$\label{twoprime} \end{thm}

Of the negative values of $d$ such that $\OQfield{d}$ is a unique factorization domain, only $-1, -2\not\equiv 1\pmod{4}$. If $d\equiv 1\pmod{4}$, then $d=4k+1$ for some integer $k$. By the Quadratic Reciprocity law, $$\left(\frac{\delta}{p}\right)=\left(\frac{d}{p}\right)=\left(\frac{p}{d}\right)(-1)^{(p-1)(d-1)/4}=\left(\frac{p}{d}\right)(-1)^{k(p-1)}=\left(\frac{p}{d}\right),$$ where $\left(\frac{a}{m}\right)$ is the Kronecker symbol. By $\eqref{inertprimes}$, we must have $\left(\frac{\delta}{p}\right)=\left(\frac{p}{d}\right)=-1$ for $p$ to be inert, so $p$ must equivalently be congruent to a quadratic non-residue modulo $|d|$. Checking this greatly reduces computational complexity from explicitly computing the symbol $(a/p)\equiv a^{(p-1)/2}\pmod{p}$.

The classification of odd inert rational primes for $d=-1$ and $d=-2$ can be derived separately with the Quadratic Reciprocity law and is well known.

Figures \eqref{fig:sub1} and \eqref{fig:sub2} were generated using the above classifications.

\begin{figure}[h]
\centering
\begin{subfigure}{.49\textwidth}
  \centering
  \includegraphics[width=.98\linewidth]{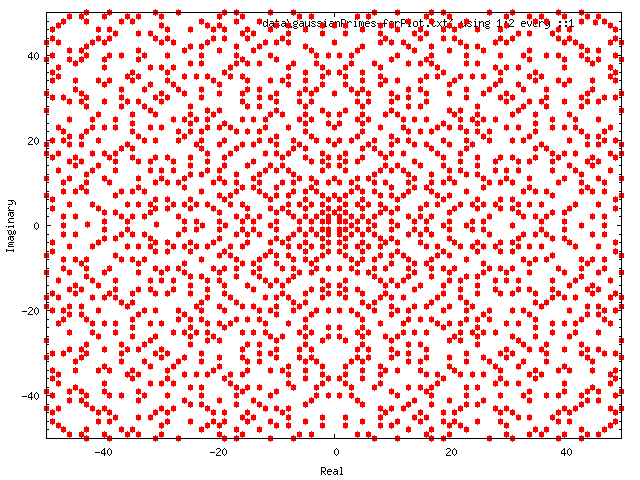}
  \caption{$\Z[i]$}
  \label{fig:sub1}
\end{subfigure}
\begin{subfigure}{.49\textwidth}
  \centering
  \includegraphics[width=.98\linewidth]{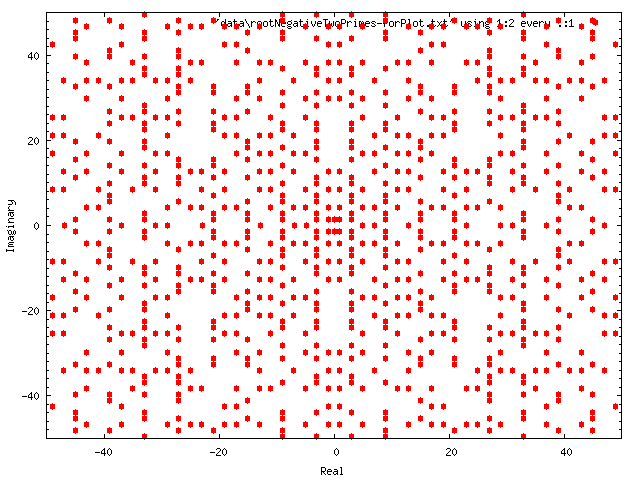}
  \caption{$\Z[\sqrt{-2}]$}
  \label{fig:sub2}
\end{subfigure}
\caption{Primes in $\Z[i]$ and $\Z[\sqrt{-2}]$}
\label{fig:test}
\end{figure}

\subsection{Moat Identification}

In this section we describe our algorithm for identifying moats.

\subsubsection{Representation of Primes}

We shall represent the primes as nodes in the Cartesian plane as per the mapping discussed below.

Suppose $\OQfield{d}=\Z[\tau]$. Let $a+b\tau\in\Z[\tau]$. In $\Z[\tau]$, a point $a+b\tau$ may be represented by the ordered pair $(a,b)$. Thus, we have the following mapping from $\Z[\tau]\to\R^2$. $$(a,b)\mapsto \left(a, b\sqrt{-d}\right) \text{ when } d=-2 \text{ and } (a,b)\mapsto \left(a-\frac{b}{2}, \frac{b}{2}\sqrt{-d}\right) \text{ otherwise}.$$ Once these transformations are applied, we can treat the primes as regular points in the Cartesian plane.

\subsubsection{Symmetry in $\OQfield{d}$}

In order to narrow the region in which we search for primes, we use the following theorem \cite{ireland}.

\begin{unnamedthm}\label{unitgroup}$\Z[i]^{\times}=\{\pm 1, \pm i\}$, $\Z[\omega]^{\times}=\{\pm 1, \pm\omega, \pm\omega^2\}$, and $\OQfield{d}^{\times}=\{\pm 1\}$ for $d=-2$ or $d<-3$, where $R^{\times}$ denotes the group of units in $R$. \end{unnamedthm}

Since multiplication by a unit and conjugation both preserve primality, the primes in $\Z[i]$ exhibit $4\times 2 = 8$ fold symmetry, those in $\Z[\omega]$ exhibits $6\times 2 = 12$ fold symmetry, and those in the other quadratic fields with unique factorization in their integers exhibit $2\times 2=4$ fold symmetry.

Thus, in $\Z[i]$, we restrict our search to regions bounded by the $x$-axis, the line $y=x$, and the $x$-boundary line $x=C$, for some constant $C$. In $\Z[\omega]$, we search in the region bounded by the $x$-axis, the line $y=\frac{1}{\sqrt{3}}x$, and the $x$-boundary line $x=C$, for some constant $C$. In the other quadratic fields we must examine an entire quadrant, so we use the region bounded by the $x$-axis, the $y$-axis, the $x$-boundary line $x=C$, and the $y$-boundary line $y=C$, for some constant $C$.

\subsubsection{Mapping from $\R^2$ to $\OQfield{d}$}

Here we describe how the previously mentioned mapping from $\OQfield{d}\to\R^2$ affects the lines of symmetry derived in the previous section.

The case of $\Z[\sqrt{-2}]$ is straightforward: the lines $x=C$ and $y=C$ in the Cartesian plane get mapped to the lines $x=C$ and $y=\frac{1}{\sqrt{2}}C$ in $\Z[\sqrt{-2}]$, respectively.

When $d\neq -2$, the map described above is equivalent to the linear transformation $\mathcal{T}:\OQfield{d}\to\R^2$ defined as $$\mathcal{T}(\vec{r})=A\cdot\vec{r}, \text{ where } A=\renewcommand\arraystretch{1}\begin{pmatrix}1 & -\tfrac{1}{2} \\ 0 & \tfrac{\sqrt{d}}{2}\end{pmatrix},$$ where $\OQfield{d}$ may be thought of as a $\Z$-module with basis $\left\{1,\frac{-1+\sqrt{d}}{2}\right\}$ if $d\equiv 1\pmod{4}$ and $\{1,\sqrt{d}\}$ otherwise. The matrix $A$ has a nonzero determinant, so $\mathcal{T}$ is invertible. $\mathcal{T}$ maps a line in $\OQfield{d}$ of the form $\vec{r}(t) = t\cdot\vec{r}+\vec{r_0}$ to the line $t\cdot(A\cdot\vec{r})+A\cdot\vec{r_0}$ in $\R^2$. Thus, we may apply the inverse linear transformation $\mathcal{T}^{-1}$ to find the lines in $\OQfield{d}$ corresponding to the lines of symmetry in the Cartesian plane. 

It follows that $\mathcal{T}^{-1}$ maps the lines $y=\frac{1}{\sqrt{3}}x$, $x=C$, and $y=0$ in the Cartesian plane to the lines $y=\frac{1}{2}x$, $y = 2x-2C$, and $y=0$ in the Eisenstein plane, respectively.

Due to the above theorem on the size of unit groups, the planes described by the other quadratic fields $\OQfield{d}$, $d<-3$ exhibit only $4$-fold symmetry. Using the inverse linear transformation $\mathcal{T}^{-1}$ developed above, the lines $x=0$, $x=C$, and $y=C$ in the Cartesian plane get mapped to the lines $y=2x$, $y=2x-2C$ and $y=\frac{2h}{\sqrt{|d|}}$ in the $\OQfield{d}$, respectively, for the remaining values of $d$.

The primes that we generate are within the symmetry-based bounds discussed above.

\subsubsection{Algorithm}

We now describe the algorithm to find moats. 

We define an undirected connected Euclidean graph with the primes representing nodes using the mapping defined in the previous section and edges represented by straight line paths between any two nodes. The cost of an edge connecting two primes $\pi_1$ and $\pi_2$ is the distance between them i.e., $\sqrt{N(\pi_1-\pi_2)}$.

Define the \textit{starting prime} as the prime with the lowest norm, the \textit{frontier prime} as the one with the largest norm discovered so far that is connected to the starting prime, the \textit{rightmost prime} as the one closest to the $x$-boundary, and the \textit{topmost prime} as the one closest to the $y$-boundary. Note that due to symmetry arguments discussed above, in $\Z[i]$ and $\Z[\omega]$, it is not necessary to consider the $y$-boundary and hence we need not define the topmost prime. 


\begin{unnamedthm}
Consider an unbounded Euclidean graph (i.e., without any $x$-boundary or $y$-boundary) containing the primes and edges as described above. Construct a minimum spanning tree (MST) \cite{lewis} containing the starting prime and the frontier prime. The moat that will be encountered in going from the starting prime to the frontier prime is the same as the second largest edge cost in the MST-based path from the starting prime to the frontier prime. Further, in a bounded graph with an $x$-boundary and/or a $y$-boundary, if the primes closest to these boundaries are at a distance from the boundary that is greater than the moat then the moat is a valid moat in the unbounded graph (without these boundaries). 

\end{unnamedthm}
\begin{proof}
In order to find the MST covering the starting prime and the frontier prime, we specifically use Kruskal's algorithm that greedily adds edges to the MST in increasing order of the edge lengths. Please refer to \cite{lewis} for a proof of Kruskal's algorithm.

In the path in the MST from the starting prime to the frontier prime, let us call the edge with the largest cost as $e_{\ell}$ and the edge with the second largest cost as $e_k$. At the point Kruskal's algorithm added $e_{\ell}$, we are guaranteed that all unused edges have a length greater than that of $e_{\ell}$.

This would appear to qualify $e_k$ to be a moat from the starting prime to the frontier prime. But, there is one more contention to reckon with. We are constraining our path based on the MST which has a goal of minimizing the overall MST cost and not necessarily minimize the moat. For this, we simply seek help from Kruskal's algorithm. As Kruskal's algorithm constructs a path from one node to another, it is guaranteed to use the shortest edges possible to build the path. Hence, an MST-based path between two nodes is guaranteed to also minimize the largest edge cost in the path between the nodes.

Hence, the existence of a $k$-moat is indicated by Kruskal's algorithm once an edge of cost greater than $k$ is added to the set of edges that will eventually form a minimum spanning tree including the starting prime. The frontier prime prior to the addition of this edge is the farthest prime that one can traverse to from the starting prime, if one is allowed to traverse edges of length $\leq k$. 

Now, for practical MST computation, we cannot handle an unbounded graph and must introduce the $x$-boundary and/or the $y$-boundary based on symmetry conditions, as discussed in previous sub-sections. Suppose the tree containing the starting prime contains edges of cost at most $c$. Once an edge of cost greater than $c$ is added by Kruskal's algorithm to the minimum spanning tree causing the discovery of a new frontier prime, then we have a possible $c$-moat. However, it is possible that if we included points outside the boundary, there could have been an edge with cost less than or equal to $c$, allowing a traversal farther out. In order to check this, we compute the distance $D_1$ from the rightmost prime to the $x$-boundary line $x=C$, and (in the fields with $2$-fold symmetry) the distance $D_2$ from the top-most prime to the $y$-boundary line $y=C$. If $D_1>c$ and $D_2>c$, we have indeed found a $c$-moat, otherwise we must extend the $x$-boundary and/or $y$-boundary and continue searching.
\end{proof}

\subsubsection{Euclidean MST}

For efficient computation of the Euclidean MST, we first compute the Delaunay triangulation, defined as the triangulation $\mathpzc{T}(\mathbf{P})$ for a set of points $\mathbf{P}$ such that there is no point in $\mathbf{P}$ that lies inside the circumcircle of any triangle in $\mathpzc{T}(\mathbf{P})$. This is checked using the fact that the point $D=(x_D, y_D)$ lies inside the circumcircle of $\triangle ABC$ ($A=(x_A, y_A), B=(x_B, y_B), C=(x_C, y_C)$) if and only if $$\det\begin{pmatrix}x_A-x_D & y_A-y_D & (x_A^2-x_D^2) + (y_A^2+y_D^2) \\[0.3em] x_B-x_D & y_B-y_D & (x_B^2-x_D^2) + (y_B^2+y_D^2)\\[0.3em] x_C-x_D & y_C-y_D & (x_C^2-x_D^2) + (y_C^2+y_D^2) \end{pmatrix} > 0.$$

In our program, we used the Java Topology Suite implementation of Delaunay triangulation~\cite{jts}. 

The proof that the Euclidean MST edges are a subset of the Delaunay triangulation edges can be found in \cite{preparata}

\subsubsection{Computation Complexity}

Computing $\mathpzc{T}(\mathbf{P})$ can be done in $O(|\mathbf{P}|\log |\mathbf{P}|)$ time. The Euclidean minimum spanning tree is a subtree of $\mathpzc{T}(\mathbf{P})$. $\mathpzc{T}(\mathbf{P})$ is guaranteed to contain $O(|\mathbf{P}|)$ edges, and thus computing the minimum spanning tree using $\mathpzc{T}(\mathbf{P})$ can be done in $O(|\mathbf{P}|\log |\mathbf{P}|)$ time with Kruskal's algorithm. This is a vast improvement over the $O(|\mathbf{P}|^2\log |\mathbf{P}|^2)$ time required by running Kruskal's algorithm on a complete graph over $\mathbf{P}$ which contains $\dbinom{|\mathbf{P}|}{2}$ edges~\cite{berg, sedgewick}.

Figure \eqref{fig:emst} displays a Euclidean minimum spanning tree (solid lines) and the Delaunay triangulation (solid and dashed lines) for Eisenstein primes.

\begin{figure}[htbp]
  \centering
  \includegraphics[width=.95\linewidth]{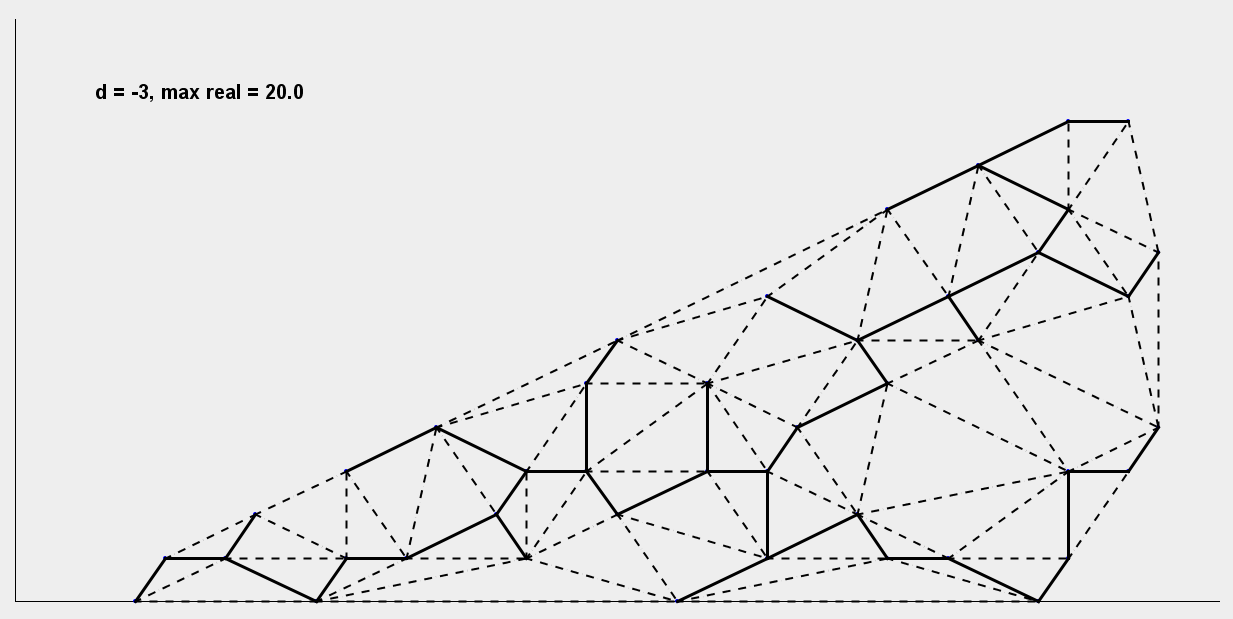}
  \caption{Euclidean MST in $\Z[\omega]$}
  \label{fig:emst}
\end{figure}

\section{Results}

We present below some of the moats found in various imaginary quadratic fields using the above algorithm.

\begin{table}	
\footnotesize
\centering 
\begin{tabular}{c c c} 
\hline 
$k$-moat & Farthest point reached & Farthest distance reached  \\ [0.8ex]  
\hline  
$1$ & $2+i$ & $2.23$  \\  
$\sqrt{2}$ & $11+4i$ & $11.705$  \\ 
$2$ & $42+17i$ & $45.310$  \\ 
$\sqrt{8}$ & $84+41i$ & $93.472$  \\ 
$\sqrt{10}$ & $976+311i$ & $1024.352$ \\
$4$ & $3297+2780i$ & $4312.61$ \\
$\sqrt{18}$ & $8174+6981i$ & $10749.355$ \\ [1ex] 
\hline 
\end{tabular} 
\caption{Moats in $\Z[i]$}
\label{table:-1} 
\end{table} 

\begin{table}
\footnotesize
\centering 
\begin{tabular}{c c c} 
\hline 
$k$-moat & Farthest point reached & Farthest distance reached  \\ [0.8ex]  
\hline  
\footnotesize$1$ & $1+\sqrt{-2}$ & $1.732$ \\  
$2$ & $3+2\sqrt{-2}$ & $4.123$  \\ 
$\sqrt{6}$ & $13+6\sqrt{-2}$ & $15.524$  \\ 
$\sqrt{12}$ & $25+42\sqrt{-2}$ & $64.444$  \\ 
$\sqrt{18}$ & $435+391\sqrt{-2}$ & $703.553$ \\
$\sqrt{22}$ & $105+653\sqrt{-2}$ & $929.432$ \\ 
$\sqrt{24}$ & $1531+555\sqrt{-2}$ & $1720.468$ \\
$\sqrt{32}$ & $1659+991\sqrt{-2}$ & $2171.737$ \\ 
$\sqrt{34}$ & $1423+2097\sqrt{-2}$ & $3289.338$ \\ [1ex] 
\hline 
\end{tabular} 
\caption{Moats in $\Z[\sqrt{-2}]$}
\label{table:-2} 
\end{table} 

\begin{table} 
\footnotesize
\centering 
\begin{tabular}{c c c} 
\hline 
$k$-moat & Farthest point reached & Farthest distance reached  \\ [0.8ex]  
\hline  
$1$ & $5+2\omega$ & $4.359$  \\  
$\sqrt{3}$ & $52+7\omega$ & $48.877$  \\ 
$2$ & $92+9\omega$ & $87.846$ \\ 
$\sqrt{7}$ & $535+49\omega$ & $512.261$ \\
$3$ & $2593+1120\omega$ & $2252.529$ \\
$\sqrt{12}$ & $20973+3518\omega$ & $19454.049$ \\[1ex] 
\hline 
\end{tabular} 
\caption{Moats in $\Z[\omega]$}
\label{table:-3} 
\end{table}

\begin{table}
\footnotesize
\centering 
\begin{tabular}{c c c} 
\hline 
$k$-moat & Farthest point reached & Farthest distance reached  \\ [0.8ex]  
\hline  
$\sqrt{2}$ & $1+2\tau$ & $2.646$  \\  
$2$ & $13+2\tau$ & $12.288$  \\ 
$\sqrt{8}$ & $37+50\tau$ & $67.224$ \\ 
$4$ & $419+326\tau$ & $501.517$ \\
$\sqrt{28}$ & $379+714\tau$ & $944.789$ \\
$\sqrt{32}$ & $17771+8308\tau$ & $17498.934$ \\[1ex] 
\hline 
\end{tabular} 
\caption{Moats in $\Z[\tau], \tau=\frac{-1+\sqrt{-7}}{2}$}
\label{table:-7} 
\end{table}

The graph below~\eqref{fig:moatplot} gives an idea of the distance reachable with a given moat size across the various imaginary quadratic fields analyzed.

\begin{center}
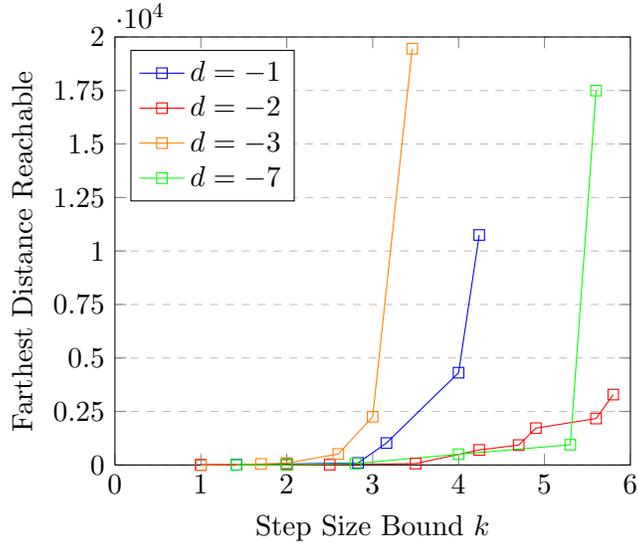
\begin{figure}
\begin{tikzpicture}
\begin{axis}[
    xlabel={Step Size Bound $k$},
    ylabel={Farthest Distance Reachable},
    xmin=0, xmax=6,
    ymin=0, ymax=20000,
    xtick={0,1,2,3,4,5,6},
    ytick={0,2500,5000,7500,10000,12500,15000, 17500, 20000},
    legend pos=north west,
    ymajorgrids=true,
    grid style=dashed,
]
 
\addplot[
    color=blue,
    mark=square,
    ]
    coordinates {
    (1,2.23)(1.414,11.7)(2,45.31)(2.83,93.47)(3.16,1024.35)(4,4312)(4.24,10749)
    };
    \addlegendentry{$d=-1$}
		
\addplot[
    color=red,
    mark=square,
    ]
    coordinates {
    (1,1.732)(2,4.12)(2.5,15.5)(3.5,64.4)(4.24,703.5)(4.7,929.4)(4.9,1720)(5.6,2171)(5.8,3289)
    };
    \addlegendentry{$d=-2$}
		
\addplot[
    color=orange,
    mark=square,
    ]
    coordinates {
    (1,4.359)(1.7,48.8)(2,87.8)(2.6,512.3)(3,2252.5)(3.46,19454)
    };
    \addlegendentry{$d=-3$}
		
\addplot[
    color=green,
    mark=square,
    ]
    coordinates {
    (1.414,2.6)(2,12.3)(2.8,67.2)(4,501.5)(5.3,944.7)(5.6,17498)
    };
    \addlegendentry{$d=-7$}

\end{axis}
\end{tikzpicture}
\caption{Moats in Imaginary Quadratic Fields}
\label{fig:moatplot}
\end{figure}
\end{center}

\section{Asymptotic Estimates on the Density of Primes}

In \cite{gethner}, the authors state that it is likely that moats of arbitrarily large size exist, thus making a walk to infinity impossible. Tsuchimura, in \cite{japan}, derives an asymptotic expression for the number of Gaussian primes in the first octant contained within the ellipse $N(x+yi)=x^2+y^2=R^2$. In this section, we derive similar expressions for $\Z[\sqrt{-2}]$ and $\Z[\omega]$, leading us to conjecture that for a given step bound $k$, we can travel farther away from the origin in $\Z[i]$ and $\Z[\omega]$ than in $\Z[\sqrt{-2}]$.

First, we state results classifying the rational primes expressible as $N(x+y\sqrt{-2})=x^2+2y^2$ and $N(x+y\omega)=x^2-xy+y^2$.

\begin{thm}\label{lattice1}An odd prime $p$ can be expressed as $x^2+2y^2$ for integers $x$ and $y$ if and only if $p\equiv 1\pmod{8}$ or $p\equiv 3\pmod{8}$. \end{thm}

\begin{thm}\label{lattice2}An odd prime $p$ can be expressed as $x^2-xy+y^2$ for integers $x$ and $y$ if and only if $p\equiv 1\pmod{3}$. \end{thm}

Proofs of these are well known and may be found in \cite{cox}.

Tsuchimura's result states that the number of Gaussian primes in the first octant with norm less than or equal to $R$ is asymptotically equal to $\frac{R^2}{4\log R}$. Using similar methods via Dirichlet's theorem and the Prime Number Theorem, along with \eqref{lattice1} and \eqref{lattice2}, we show that this expression holds for the number of primes in all $\OQfield{d}$ with unique factorization.

Let $(\Z/m\Z)^{\times} = \{u_1, u_2,\ldots, u_{\varphi(m)}\}$ denote the unit group modulo $m$, where $\varphi$ is Euler's totient function. Define $\pi_{u_i}(x)$ to be the number of primes $p\le x$ such that $p\equiv u_i\pmod{m}$. Then, Dirichlet's theorem states that $$\pi_{u_1}(x)\sim\pi_{u_2}(x)\sim\cdots\sim\pi_{u_{\varphi(m)}}(x)\sim\frac{1}{\varphi(m)}\cdot\pi(x),$$ where $\pi(x)$ is the prime-counting function. The Prime Number Theorem states that $\pi(x)\sim\frac{x}{\log x}$.

\begin{thm}\label{numprimes-2}The number of primes $x+y\sqrt{-2}$ in $\Z[\sqrt{-2}]$ with $x> 0$, $y\ge 0$ and $N(x+y\sqrt{-2})\le R^2$ can be estimated as $\frac{R^2}{4\log R}$. \end{thm}

\begin{proof}Due to \eqref{normprime}, \eqref{inertprimes}, and \eqref{lattice1}, primes in $\Z[\sqrt{-2}]$ are elements with norm $p$, where $p$ is a rational prime of the form $8k+1$ or $8k+3$ for $k\in\Z$, and rational primes $q$ of the form $8\ell+5$ or $8\ell+7$ for $\ell\in\Z$. By Dirichlet's theorem, $$\pi_1(x)\sim\pi_3(x)\sim\pi_5(x)\sim\pi_7(x)\sim\frac{\pi(x)}{4}.$$ Thus, the number of primes may be estimated as $$\pi_1(R^2)+\pi_3(R^2)+\pi_5(R)+\pi_7(R)\sim\frac{\pi(R^2)}{2}\sim\frac{R^2}{4\log R},$$ by the Prime Number Theorem. \end{proof}

\begin{cor}\label{dense1} The density of primes $x+y\sqrt{-2}$ in $\Z[\sqrt{-2}]$ with $x>0$, $y\ge 0$ and $N(x+y\sqrt{-2})\le R^2$ can be estimated as $\frac{\sqrt{2}}{\pi\log R}$. \end{cor}

\begin{proof} The area of $\frac{1}{4}$ of the ellipse $x^2+2y^2=R^2$ is $A=\frac{\pi R^2}{4\sqrt{2}}$, so the density of primes is $\frac{1}{A}\cdot\frac{R^2}{4\log R}=\frac{\sqrt{2}}{\pi\log R}$. \end{proof}

\begin{thm}\label{numprimes-3}The number of primes $x+y\omega$ in $\Z[\omega]$ with $x>0$, $y\ge 0$ and $N(x+y\omega)\le R^2$ can be estimated as $\frac{R^2}{4\log R}$. \end{thm}

\begin{proof}Due to \eqref{normprime}, \eqref{inertprimes}, and \eqref{lattice2}, primes in $\Z[\omega]$ are elements with norm $p$, where $p$ is a rational prime of the form $3k+1$ for $k\in\Z$, and rational primes $q$ of the form $3\ell+2$, for $\ell\in\Z$. By Dirichlet's theorem, $$\pi_1(x)\sim\pi_2(x)\sim\frac{\pi(x)}{2}.$$ Thus, the number of primes may be estimated as $$\pi_1(R^2)+\pi_2(R)\sim\frac{\pi(R^2)}{2}\sim\frac{R^2}{4\log R},$$ by the Prime Number Theorem. \end{proof}

\begin{cor}\label{dense2} The density of primes $x+y\omega$ in $\Z[\omega]$ with $x>0$, $y\ge 0$ and $N(x+y\omega)\le R^2$ can be estimated as $\frac{3\sqrt{3}}{2\pi\log R}$. \end{cor}

\begin{proof} The area of $\frac{1}{12}$ of the ellipse $x^2-xy+y^2=R^2$ is $A=\frac{\pi R^2}{6\sqrt{3}}$, so the density of primes is $\frac{1}{A}\cdot\frac{R^2}{4\log R}=\frac{3\sqrt{3}}{2\pi\log R}$. \end{proof}

Note that if $\Q(\sqrt{d})$ is an imaginary quadratic field (other than $\Q(i)$ and $\Q(\sqrt{-2})$) and its ring of integers $\OQfield{d}$ is a unique factorization domain, we can make a more general statement due to results from \cite{cox}. In general, a prime $p$ is expressible as a quadratic form with discriminant $\delta$ equal to that of the discriminant of the unique factorization domain $\OQfield{d}$ if and only if $\left(\frac{p}{\delta}\right) = 1$. Note that $\delta$ is prime, so we may write $(\Z/\delta\Z)^{\times}=\{1,2,\ldots,\delta-1\}$. Now, let $A$ denote the set of quadratic residues in $(\Z/\delta\Z)^{\times}$ and let $B$ denote the set of quadratic non-residues. Then, $|A|=|B|=\frac{\delta-1}{2}$.

This leads us to state the following theorem.

\begin{thm}\label{numprimes-d} If $\OQfield{d} = \Z[\tau]$ is a unique factorization domain with $d<0$, the number of primes $x+y\tau$ with $x>0$, $y\ge 0$ and $N(x+y\tau)\le R^2$ can be estimated as $\frac{R^2}{4\log R}$. \end{thm}

\begin{proof}Due to \eqref{normprime}, \eqref{inertprimes}, and the above, primes in $\Z[\tau]$ are elements with norm $p$, where $\left(\frac{p}{\delta}\right)=1$, and rational primes $q$ such that $\left(\frac{q}{\delta}\right)=-1$. By Dirichlet's theorem, $$\pi_1(x)\sim\pi_2(x)\sim\cdots\sim\pi_{\delta-1}(x)\sim\frac{\pi(x)}{\delta-1}.$$ Thus, the number of primes may be estimated as $$\sum_{u_i\in A}\pi_{u_i}(R^2)+\sum_{u_j\in B}\pi_{u_j}(R)\sim \frac{\delta-1}{2}\cdot\frac{\pi(R^2)}{\delta-1}\sim\frac{R^2}{4\log R},$$ by the Prime Number Theorem. \end{proof}

Due to Tsuchimura, the density of primes in $x+yi\in\Z[i]$ with $x>0$, $y\ge 0$, and $N(x+yi)\le R^2$ is $\frac{2}{\pi\log R}$. We have $$\frac{\sqrt{2}}{\pi\log R}<\frac{2}{\pi\log R}<\frac{3\sqrt{3}}{2\pi\log R},$$ so by \eqref{dense1} and \eqref{dense2}, we may conjecture that for a given step bound $k$, one can travel the farthest in $\Z[\omega]$, the next farthest in $\Z[i]$ and the least in $\Z[\sqrt{-2}]$. Note that as $|d|$ becomes larger, the area of the ellipse $x^2-xy-\frac{d-1}{4}y^2=R^2$ decreases. Thus, by \eqref{numprimes-d}, we may also conjecture that the farthest distance reachable with a given step bound $k$ increases as $d$ decreases with the constraint that $\OQfield{d}$ is a unique factorization domain.

\section{Conclusion and Future Work}

In this paper we explored new imaginary quadratic fields exhibiting unique factorization for moats. We developed mappings between $\OQfield{d}$ and $\R^2$, and used symmetry arguments to reduce the set of primes that need to be explored to find moats. We also proposed a new and efficient algorithm by mapping this problem to the Euclidean minimum spanning tree problem with a time complexity of $O(|\mathbf{P}|\log |\mathbf{P}|)$ time, where $\mathbf{P}$ denotes the set of primes we are searching through. While the underlying algorithms -- including those to identify the triangulation and the minimum spanning tree to find a path with the minimum maximal edge length -- are known, our algorithm appears to be the first to apply it explicitly to find all moats, in increasing order of length, along with the associated farthest prime. Further to deal with an underlying infinite unbounded graph with infinite nodes (primes), the algorithm uses symmetry and boundaries to correctly and efficiently identify moats.

It uses a Euclidean minimum spanning tree (EMST) to find moats. It uses Delaunay triangulation and Kruskal's algorithm to solve the EMST thereby bringing the overall time complexity to $O(|\mathbf{P}|\log |\mathbf{P}|)$  where $\mathbf{P}$ is the set of primes considered.


We also calculated the density of primes in $\Z[i]$, $\Z[\sqrt{-2}]$, and $\Z[\omega]$. Based on this and our computational moat results, for a given step size $k$, the farthest distance reachable appears to be greatest in $\Z[\omega]$, second greatest in $\Z[i]$, and least in $\Z[\sqrt{-2}]$. $\Z\left[\frac{-1+\sqrt{-7}}{2}\right]$ is harder to analyze, since the computational results do not suggest any discernible trend.


Some directions for future development include the following:

\begin{itemize}

\item Compare moat results for the imaginary quadratic fields with $d=-1,-2,-3,-7,-11$ to those with $d=-19,-43,-67,-163$, the former group containing those that are both Euclidean (has a division algorithm) and unique factorization domains (UFDs) and the latter containing those that are UFDs but non-Euclidean.

\item Investigate imaginary quadratic fields that are not necessarily Euclidean or UFDs. Here, one would have to look at the splitting of prime ideals and represent these geometrically.

\item Investigate real quadratic fields ($d>0$). Here, the geometric interpretation is less intuitive, since these contain infinitely many units.	

\item Investigate a more general formulation of the moat problem that is not restricted to prime elements: if $S\subseteq\Z^+$ and $T=\{\alpha\in\OQfield{d}, d<0 : N(\alpha)\in S\}$, how do moats in $T$ behave?

\item Exploit computationally hard problems for cryptography. For example, given an arbitrary prime in $\OQfield{d}$, finding the associated moat appears hard.

\end{itemize}

\section{Acknowledgements}

I would like to thank Dr. Daniel Shapiro and the Ross Mathematics program at the Ohio State University. As a result of my two summers at this pivotal program as a student and a Junior Counselor, I have been inspired to delve deeply into various exciting fields in mathematics. I would also like to thank Dr. Simon Rubinstein-Salzedo for reviewing my paper and giving me invaluable feedback and suggestions.

Finally, I would like to thank my family for unconditionally supporting my explorations.

\end{document}